\documentclass[12pt,reqno]{article}

\usepackage[usenames]{color}
\usepackage{amssymb}
\usepackage{graphicx}
\usepackage{amscd}

\usepackage{lscape}
\usepackage{tikz}
\usepackage{tikz-cd}
\usepackage{tkz-fct}
\usepackage{pgfplots}
\usetikzlibrary{matrix}
\usetikzlibrary{fit,shapes}
\usetikzlibrary{positioning, calc}
\usepackage{colortbl}

\usepackage{amsthm}
\newtheorem{theorem}{Theorem}

\newtheorem{proposition}[theorem]{Proposition}
\newtheorem{corollary}[theorem]{Corollary}

\theoremstyle{definition}

\newtheorem{example}[theorem]{Example}

\usepackage[colorlinks=true,
linkcolor=webgreen, filecolor=webbrown,
citecolor=webgreen]{hyperref}

\definecolor{webgreen}{rgb}{0,.5,0}
\definecolor{webbrown}{rgb}{.6,0,0}

\usepackage{color}
\usepackage{float}
\usepackage{graphics,amsmath,amssymb}
\usepackage{amsfonts}
\usepackage{latexsym}
\usepackage{epsf}

\newcommand{\seqnum}[1]{\href{http://oeis.org/#1}{\underline{#1}}}

\begin{document}

\begin{center}
\vskip 1cm{\LARGE\bf Notes on Riordan arrays and lattice paths} \vskip 1cm \large
Paul Barry\\
School of Science\\
South East Technological University\\
Ireland\\
\href{mailto:pbarry@wit.ie}{\tt pbarry@wit.ie}
\end{center}
\vskip .2 in

\begin{abstract} In this note, we explore links between Riordan arrays and lattice paths. We begin by describing Riordan arrays, and some of their generalizations, including rectifications and triangulations. We the consider Riordan array links to lattice paths with  steps of type $(a,b)$, where $a$ and $b$ are nonnegative. We consider common Riordan arrays that are linked to lattice paths, as well as showing links between almost Riordan arrays and lattice paths. We then consider lattice paths with  step sets that include downward steps, and show how the $A$-matrix characterization of Riordan arrays plays a key role in analysing corresponding Riordan arrays.\end{abstract}

\section{Preliminaries on Riordan arrays}
A Riordan array $(g(x), f(x))$ \cite{book, book2} is defined by two power series $g(x)$ and $f(x)$, were
$$g(x) \in \mathcal{F}_0 = \{\sum_{n=0}^{\infty} a_n x^n\,|\, a_0 \ne 0, a_n \in R\},$$ and
$$f(x) \in \mathcal{F}_1 =\{\sum_{n=1}^{\infty} a_n x^n\,|\, a_0 = 0, a_1 \ne 0, a_n \in R\} $$
where $R$ is any ring for which the following operations are valid (for instance, $R$ could be the ring of integers $\mathbb{Z}$, or any of the fields
$\mathbb{Q}, \mathbb{R}, \mathbb{C}$). We have a product
$$(g(x), f(x))\cdot (u(x), v(x))=(g(x)u(f(x)), v(f(x)),$$ with identity $(1,x)$, and an inverse
$$(g(x), f(x))^{-1} = \left(\frac{1}{g(\bar{f}(x))}, \bar{f}(x)\right),$$ where
$\bar{f}(x)=f^{\langle -1 \rangle}(x)$ is the compositional inverse of $f$ (thus $\bar{f}(x)$ is the solution $u(x)$ of the equation $f(u(x))=x$ for which we have $u(0)=0$).
With these operations the set of Riordan arrays (over $R$) becomes a group. There is a matrix representation of this group. To each Riordan array $(g(x), f(x))$ we associate the matrix $(a_{n,k})$, with
$a_{n,k} \in R$, where
$$a_{n,k}=[x^n] g(x)f(x)^k,$$ where $[x^n]$ \cite{Method} is the functional operating on power series that extracts the coefficient of $x^n$. The matrix $(a_{n,k})$ will then be an invertible lower-triangular matrix, and the product rule in the group becomes ordinary matrix multiplication. The product rule embodies the ``fundamental theorem of Riordan arrays'', which describes how a Riordan array $(g(x), f(x))$ operates on a power series $a(x)$. We have
$$(g(x), f(x))\cdot a(x)=g(x)a(f(x)).$$
For instance, the row sums of a Riordan array $(g(x), f(x))$ have their generating function given by
$$(g(x), f(x))\cdot \frac{1}{1-x}= \frac{g(x)}{1-f(x)}.$$
A Riordan array of the form $(g(x),xg(x))$ is called a Bell matrix. Such matrices form a subgroup of the Riordan group.
\begin{example}
The Riordan array $\left(\frac{1}{1-x}, \frac{x}{1-x}\right)$ is represented by the familiar Pascal's triangle matrix $\left(\binom{n}{k}\right)$ (also called the \emph{binomial matrix}) that begins
$$\left(
\begin{array}{ccccccc}
 1 & 0 & 0 & 0 & 0 & 0 & 0 \\
 1 & 1 & 0 & 0 & 0 & 0 & 0 \\
 1 & 2 & 1 & 0 & 0 & 0 & 0 \\
 1 & 3 & 3 & 1 & 0 & 0 & 0 \\
 1 & 4 & 6 & 4 & 1 & 0 & 0 \\
 1 & 5 & 10 & 10 & 5 & 1 & 0 \\
 1 & 6 & 15 & 20 & 15 & 6 & 1 \\
\end{array}
\right).$$
This is a special example, in that it is ``centrally symmetric''. It is sequence A007318 in Sloane's On-Line Encyclopedia of Integer Sequences \cite{SL1, SL2}. Another centrally symmetric or Pascal-like Riordan array is given by
$$\left(\frac{1}{1-x}, \frac{x(1+x)}{1-x}\right),$$ whose matrix representation begins
$$\left(
\begin{array}{ccccccc}
 1 & 0 & 0 & 0 & 0 & 0 & 0 \\
 1 & 1 & 0 & 0 & 0 & 0 & 0 \\
 1 & 3 & 1 & 0 & 0 & 0 & 0 \\
 1 & 5 & 5 & 1 & 0 & 0 & 0 \\
 1 & 7 & 13 & 7 & 1 & 0 & 0 \\
 1 & 9 & 25 & 25 & 9 & 1 & 0 \\
 1 & 11 & 41 & 63 & 41 & 11 & 1 \\
\end{array}
\right).$$ This is the Delannoy triangle, \seqnum{A008288} (we shall distinguish between the Delannoy triangle and the Delannoy square later). All Riordan arrays that are centrally symmetric are of the form
$\left(\frac{1}{1-x}, \frac{x(1+rx)}{1-x}\right)$.
\end{example}
As two-dimensional objects, Riordan arrays have a bivariate generating functions $A(x,y)=\sum_{n,k} a_{n,k}x^n y^k$. We have in fact that
$$A(x,y)=(g(x), f(x))\cdot \frac{1}{1-yx}=\frac{g(x)}{1-yf(x)}.$$ The first column will have generating function $A(x,0)$, while the row sums will have generating function $A(x,1)$.
\begin{example}
For Pascal's triangle, we have
$$A(x,y)=\frac{\frac{1}{1-x}}{1-y\frac{x}{1-x}}=\frac{1}{1-x-xy},$$ while for the Delannoy triangle, we have
$$A(x,y)=\frac{\frac{1}{1-x}}{1-y\frac{x(1+x)}{1-x}}=\frac{1}{1-x-xy-x^2y}.$$
\end{example}
We shall use the notation
$$\mathcal{G}(g(x), f(x))=A(x,y)$$ where appropriate in the following.

By the \emph{rectification} of a Riordan array we shall understand the square matrix with bivariate generating function
$$ \frac{g(x)}{1-y \frac{f(x)}{x}}.$$ We can represent this as $\left(g(x), \frac{f(x)}{x}\right)$. Using the rules of the $[x^n]$ functional, we have
$$[x^n] g(x) \left(\frac{f(x)}{x}\right)^k = [x^{n+k}] g(x)f(x)^k,$$ and hence this ``rectified'' matrix has general element $a_{n+k,k}$.
\begin{example}
The rectification of the binomial triangle (Pascal's triangle) is given by $\left(\binom{n+k}{k}\right)$, which is the square matrix that begins
$$\left(
\begin{array}{ccccccc}
 1 & 1 & 1 & 1 & 1 & 1 & 1 \\
 1 & 2 & 3 & 4 & 5 & 6 & 7 \\
 1 & 3 & 6 & 10 & 15 & 21 & 28 \\
 1 & 4 & 10 & 20 & 35 & 56 & 84 \\
 1 & 5 & 15 & 35 & 70 & 126 & 210 \\
 1 & 6 & 21 & 56 & 126 & 252 & 462 \\
 1 & 7 & 28 & 84 & 240 & 462 & 924 \\
\end{array}
\right).$$
This matrix has generating function
$$A\left(x, \frac{y}{x}\right)=\frac{1}{1-x-y},$$ where
$$A(x,y)=\frac{1}{1-x-xy}$$ is the generating function of Pascal's triangle.
In a similar way, the rectification of the Delannoy triangle gives us the Delannoy square which has generating function
$$\frac{1}{1-x-y-xy}.$$
\end{example}
The (vertical) \emph{stretching} of a Riordan array $(g(x), f(x))$, represented by $(g(x), xf(x))$, is the matrix with generating function
$$\frac{g(x)}{1-yxf(x)}.$$ The row sums of this matrix are the \emph{diagonal sums} of the original Riordan array $(g(x), f(x))$. This means that the generating function of the diagonal sums of a Riordan array $(g(x), f(x))$ is given by $\frac{g(x)}{1-xf(x)}$.
\begin{example} The stretching of Pascal's triangle $\left(\frac{1}{1-x}, \frac{x}{1-x}\right)$ is given by $\left(\frac{1}{1-x}, \frac{x^2}{1-x}\right)$.
This is the matrix $\left(\binom{n-k}{k}\right)$ which begins
$$\left(
\begin{array}{ccccccc}
 1 & 0 & 0 & 0 & 0 & 0 & 0 \\
 1 & 0 & 0 & 0 & 0 & 0 & 0 \\
 1 & 1 & 0 & 0 & 0 & 0 & 0 \\
 1 & 2 & 0 & 0 & 0 & 0 & 0 \\
 1 & 3 & 1 & 0 & 0 & 0 & 0 \\
 1 & 4 & 3 & 0 & 0 & 0 & 0 \\
 1 & 5 & 6 & 1 & 0 & 0 & 0 \\
\end{array}
\right).$$
The row sums in this case are the positive Fibonacci numbers with generating function
$$\frac{\frac{1}{1-x}}{1-\frac{x^2}{1-x}}=\frac{1}{1-x-x^2}.$$
\end{example}
The \emph{reversal} of a lower-triangular matrix with generating function $A(x,y)$ is the lower-triangular matrix with generating function
$$A\left(xy, \frac{1}{y}\right).$$ If the matrix in question is a Riordan array $(g(x),f(x))$ then the elements of the reversal are given by
$$a_{n,n-k}=[x^n] g(x)f(x)^{n-k}.$$

Note that we only consider ordinary generating functions. The Riordan arrays we consider are sometimes called ordinary Riordan arrays, to distinguish them from ones that use, for instance, exponential generating functions.

\section{Characterizations of Riordan arrays}
An important feature of Riordan arrays \cite{book, book2, Survey, SGWW} is that they have a sequence characterization. Specifically, a lower-triangular array $(t_{n,k})_{0 \le n,k \le \infty}$ is a Riordan array if and only if there exists a sequence $a_n, n\ge 0$, such that
$$t_{n,k}=\sum_{i=0}^{\infty}a_i t_{n-1,k-1+i}.$$
This sum is actually a finite sum, since the matrix is lower-triangular. For such a matrix $M$, we let  $$P=M^{-1}\overline{M},$$ where $\overline{M}$ is the matrix $M$ with its top row removed. Then the matrix $P$ has the form that begins
$$\left(
\begin{array}{ccccccc}
 z_0 & a_0 & 0 & 0 & 0 & 0 & 0 \\
 z_1 & a_1 & a_0 & 0 & 0 & 0 & 0 \\
 z_2 & a_2 & a_1 & a_0 & 0 & 0 & 0 \\
 z_3 & a_3 & a_2 & a_1 & a_0 & 0 & 0 \\
 z_4 & a_4 & a_3 & a_2 & a_1 & a_0 & 0 \\
 z_5 & a_5 & a_4 & a_3 & a_2 & a_1 & a_0 \\
 z_6 & a_6 & a_5 & a_4 & a_3 & a_2 & a_1 \\
\end{array}
\right).$$

Here, $z_0,z_1,z_2,\ldots$ is an ancillary sequence which exists for any lower-triangular matrix. The matrix $P$ is called the \emph{production matrix} of the Riordan matrix $M$. The sequence characterization of renewal arrays was first described by Rogers \cite{He_A, Rogers}.
An alternative approach to the sequence characterization of a Riordan array is to use a matrix characterization \cite{He_M, Merlini}.  One form that such a matrix characterization may take is the following \cite{He_M, Merlini}.
\begin{theorem} A lower-triangular array $(t_{n,k})_{0 \le n,k \le \infty}$ is a Riordan array if and only if there exists another array $A=(a_{i,j})_{i,j \in \mathbb{N}_0}$ with $a_{0,0} \ne 0$, and a sequence $(\rho_j)_{j \in \mathbb{N}_0}$ such that
$$t_{n+1,k+1}=\sum_{i \ge 0} \sum_{j \ge 0} a_{i,j} t_{n-i, k+j} + \sum_{j \ge 0} \rho_j t_{n+1,k+j+2}.$$
\end{theorem}
As we have seen, the power series definition of a Riordan array is as follows. A Riordan array is defined by a pair of power series, $g(x)$ and $f(x)$, where
$$g(x)=g_0 + g_1 x + g_2 x^2+ \cdots, \quad g_0 \ne 0,$$ and
$$f(x)=f_1 x + f_2 x^2+ f_3 x^3+\cdots, \quad f_0=0 \text{ and } f_1 \ne 0.$$
We then have
$$t_{n,k}=[x^n] g(x)f(x)^k,$$ where $[x^n]$ is the functional that extracts the coefficient of $x^n$.
The relationship between $f(x)$ and the pair $(A, \rho)$ is the following.
$$\frac{f(x)}{x}=\sum_{i \ge 0} x^i R^{(i)}(f(x))+\frac{f(x)^2}{x}\rho(f(x)),$$ where
$R^{(i)}$ is the generating series of the $i$-th row of $A$, and $\rho(x)$ is the generating series of the sequence $\rho_n$.

\section{Lattice paths and Riordan arrays}
Following \cite{Rook}, we let $S \subset \mathbb{N} \times \mathbb{N}$ where $\mathbb{N}$ is the set of nonnegative integers, with $(0,0) \notin S$. An $S$-path from the point $(n_0,k_0)$ to the point $(n,k)$ is a sequence of pairs $(a_1,b_1), (a_2,b_2), \ldots,(a_r,b_r)$ such that
$$(n_0+a_1+a_2+\cdots + a_r, k_0+b_1+b_2+\cdots+b_r)=(n,k).$$ Then
$$A(x,y)=\frac{1}{1-\sum_{(a,b) \in S} x^a y^b}$$ is the generating function of the matrix $(a_{n,k})$ where
$$A(x,y)=\sum_{n,k} a_{n,k} x^n y^k,$$ and $a_{n,k}$ counts the number of $S$-paths from $(0,0)$ to $(n,k)$. The condition $(0,0) \notin S$ ensures that $a_{n,k}$ is finite.
We are interested in the case where $y$ appears in the denominator to the power of $1$ only.
\begin{proposition} Let $$A(x,y)=\frac{1}{1-\sum_{(a,b) \in S} x^a y^b}$$ with $(0,0) \notin S$.
If $$\sum_{(a,b) \in S} x^a y^b=\sum_{i=1}^r \alpha_i x^i +y x \sum_{j=0}^s \beta_j x^j, \quad \beta_0 \ne 0,$$
then the matrix $\left(a_{n,k}\right)$ is a Riordan array.\end{proposition}
\begin{proof}
We have
\begin{align*}\frac{1}{1-\sum_{i=1}^r \alpha_i x^i -y x \sum_{j=0}^s \beta_j x^j}&=\frac{\frac{1}{1-\sum_{i=1}^r \alpha_i x^i}}{1-xy\frac{\sum_{j=0}^s \beta_j x^j}{1-\sum_{i=1}^r \alpha_i x^i}}\\
&=\mathcal{G}\left(\frac{1}{1-\sum_{i=1}^r \alpha_i x^i}, \frac{x \sum_{j=0}^s \beta_jx^j}{1-\sum_{i=1}^r \alpha_i x^i}\right).\end{align*}
\end{proof}

\begin{figure}
\begin{center}
\begin{tikzpicture}
\draw(0,0)--(6,0);
\draw(0,0)--(6,6);
\draw(1,0)--(1,1.2);
\draw(2,0)--(2,2.2);
\draw(3,0)--(3,3.2);
\draw(4,0)--(4,4.2);
\draw(5,0)--(5,5.2);
\draw(6,0)--(6,6.2);

\draw(1,1)--(6.2,1);
\draw(2,2)--(6.2,2);
\draw(3,3)--(6.2,3);
\draw(4,4)--(6.2,4);
\draw(5,5)--(6.2,5);

\draw(1,0)--(6,5);
\draw(2,0)--(6,4);
\draw(3,0)--(6,3);
\draw(4,0)--(6,2);
\draw(5,0)--(6,1);

\draw[red, thick, postaction={decorate},
      decoration={markings, mark=at position 0.5 with {\arrow{>}}}]
      (0,0) -- (1,1);

\draw[red, thick, postaction={decorate},
      decoration={markings, mark=at position 0.5 with {\arrow{>}}}]
      (1,1) -- (2,2);

\draw[red, thick, postaction={decorate},
      decoration={markings, mark=at position 0.5 with {\arrow{>}}}]
      (2,2) -- (3,3);

\draw[red, thick, postaction={decorate},
      decoration={markings, mark=at position 0.5 with {\arrow{>}}}]
      (1,0) -- (2,1);

\draw[red, thick, postaction={decorate},
      decoration={markings, mark=at position 0.5 with {\arrow{>}}}]
      (2,0) -- (3,1);

\draw[red, thick, postaction={decorate},
      decoration={markings, mark=at position 0.5 with {\arrow{>}}}]
      (2,1) -- (3,2);
      
\draw[red, thick, postaction={decorate},
      decoration={markings, mark=at position 0.5 with {\arrow{>}}}]
      (3,1) -- (4,2);

\draw[blue, thick, postaction={decorate},
      decoration={markings, mark=at position 0.5 with {\arrow{>}}}]
      (0,0) -- (1,0);

\draw[blue, thick, postaction={decorate},
      decoration={markings, mark=at position 0.5 with {\arrow{>}}}]
      (1,0) -- (2,0);

\draw[blue, thick, postaction={decorate},
      decoration={markings, mark=at position 0.5 with {\arrow{>}}}]
      (2,0) -- (3,0);

\draw[blue, thick, postaction={decorate},
      decoration={markings, mark=at position 0.5 with {\arrow{>}}}]
      (1,1) -- (2,1);

\draw[blue, thick, postaction={decorate},
      decoration={markings, mark=at position 0.5 with {\arrow{>}}}]
      (2,1) -- (3,1);

\draw[blue, thick, postaction={decorate},
      decoration={markings, mark=at position 0.5 with {\arrow{>}}}]
      (2,2) -- (3,2);

\draw[blue, thick, postaction={decorate},
      decoration={markings, mark=at position 0.5 with {\arrow{>}}}]
      (3,2) -- (4,2);
      
\draw[green, thick, postaction={decorate},
      decoration={markings, mark=at position 0.5 with {\arrow{>}}}]
      (0,0) -- (2,1);

\draw[green, thick, postaction={decorate},
      decoration={markings, mark=at position 0.5 with {\arrow{>}}}]
      (1,0) -- (3,1);

\draw[green, thick, postaction={decorate},
      decoration={markings, mark=at position 0.5 with {\arrow{>}}}]
      (1,1) -- (3,2);
      
\draw[green, thick, postaction={decorate},
      decoration={markings, mark=at position 0.5 with {\arrow{>}}}]
      (2,1) -- (4,2);

\node at (0,-.3) {1};
\node at (1,1.3) {1};
\node at (2,2.3) {1};
\node at (3,3.3) {1};
\node at (4,4.3) {1};
\node at (5,5.3) {1};
\node at (6,6.3) {1};

\node at (1,-.3) {1};
\node at (2,-.3) {1};
\node at (3,-.3) {1};
\node at (4,-.3) {1};
\node at (5,-.3) {1};

\node at (2,1.3) {3};
\node at (3,1.3) {5};

\node at (3,2.3) {5};

\node at (4,1.3) {7};
\node at (4,2.3) {13};
\node at (4,3.3) {7};

\node at (5,1.3) {9};
\node at (5,2.3) {25};
\node at (5,3.3) {25};
\node at (5,4.3) {9};
\end{tikzpicture}
\end{center}
\caption{The Delannoy triangle $\left(\frac{1}{1-x},\frac{x(1+x)}{1-x}\right)$ as a path matrix}\label{Del}
\end{figure}
\begin{example} We consider the Delannoy triangle $\left(\frac{1}{1-x},\frac{x(1+x)}{1-x}\right)$, which has its generating function given by
$$\frac{1}{1-x-xy-x^2y}.$$ We conclude that this counts paths from $(0,0)$ to $(n,k)$ with steps $\{(1,0), (1,1), (2,1)\}$. See Figure \ref{Del}.

\end{example}
\begin{example}
We consider the Riordan array $(g(x), f(x))$ given by
$$\left(\frac{1}{1-x-x^2}, \frac{x(1+x)}{1-x-x^2}\right).$$
We have
$$\mathcal{G}(g(x),f(x))=\frac{1}{1-x-x^2-xy-x^2y}.$$
Hence this Riordan array enumerates ``left-factors'' of lattice paths with step set
$$S=\{(1,0), (2,0), (1,1), (2,1)\}.$$
This matrix begins
$$\left(
\begin{array}{ccccccc}
 1 & 0 & 0 & 0 & 0 & 0 & 0 \\
 1 & 1 & 0 & 0 & 0 & 0 & 0 \\
 2 & 3 & 1 & 0 & 0 & 0 & 0 \\
 3 & 7 & 5 & 1 & 0 & 0 & 0 \\
 5 & 15 & 16 & 7 & 1 & 0 & 0 \\
 8 & 30 & 43 & 29 & 9 & 1 & 0 \\
13 & 58 & 104 & 95 & 46 & 11 & 1 \\
\end{array}
\right).$$ The left-factors are the paths that go from $(0,0)$ to $(n,k)$ for general $(n,k)$.
The row sums of this matrix begin
$$1, 2, 6, 16, 44, 120, 328, 896, 2448, 6688, 18272, \ldots,$$ with generating function
$$\frac{1}{1-2x-2x^2}.$$ This sequence \seqnum{A002605}, with general term $\sum_{k=0}^{\lfloor \frac{n}{2} \rfloor} \binom{n-k}{k}2^{n-k}$,  thus counts left factors of $S$-paths.
The rectification of this array will have generating function
$$\frac{1}{1-x-y-x^2-xy}.$$
This square matrix thus enumerates paths in the positive quadrant from $(0,0)$ to $(n,k)$ with step set
$$\{(1,0),(0,1),(2,0),(1,1)\}.$$
The reversal of the Riordan array $\left(\frac{1}{1-x-x^2}, \frac{x(1+x)}{1-x-x^2}\right)$ begins
$$\left(
\begin{array}{ccccccc}
 1 & 0 & 0 & 0 & 0 & 0 & 0 \\
 1 & 1 & 0 & 0 & 0 & 0 & 0 \\
 1 & 3 & 2 & 0 & 0 & 0 & 0 \\
 1 & 5 & 7 & 3 & 0 & 0 & 0 \\
 1 & 7 & 16 & 15 & 5 & 0 & 0 \\
 1 & 9 & 29 & 43 & 30 & 8 & 0 \\
 1 & 11 & 46 & 95 & 104 & 58 & 13 \\
\end{array}
\right).$$ This matrix has generating function
$$\frac{1}{1-x-xy-x^2y-x^2y^2}.$$
It enumerates paths from $(0,0)$ to $(n,k)$ with steps from the set $\{(1,0),(1,1),(2,1),(2,2)\}$.
\end{example}
Note that in the foregoing, the coefficients $\alpha_i$ and $\beta_j$ have values of $\pm 1$, indicating membership or not of $S$. However, by allowing each step type to have different colors (or weights), we can allow for $\alpha_i, \beta_j \in \mathbb{N}$.
\begin{example} \textbf{The Delannoy triangle and square}. We have seen that the Delannoy triangle has its generating function given by
$$\frac{1}{1-x-xy-x^2y}.$$
This triangle counts lattice paths from $(0,0)$ to $(n,k)$ with step set $\{(1,0),(1,1),(2,1)\}$. These paths do not go above the diagonal $y=x$. The generating function of the Delannoy square is
$$\frac{1}{1-x-y-xy},$$ and hence this square matrix counts lattice paths in the positive quadrant from $(0,0)$ to $(n,k)$ with step set $\{(1,0),(0,1),(1,1)\}$. The Delannoy square begins
$$\left(
\begin{array}{ccccccc}
 1 & 1 & 1 & 1 & 1 & 1 & 1 \\
 1 & 3 & 5 & 7 & 9 & 11 & 13 \\
 1 & 5 & 13 & 25 & 41 & 61 & 85 \\
 1 & 7 & 25 & 63 & 129 & 231 & 377 \\
 1 & 9 & 41 & 129 & 321 & 681 & 1289 \\
 1 & 11 & 61 & 231 & 681 & 1683 & 3653 \\
 1 & 13 & 85 & 377 & 1289 & 3653 & 8989 \\
\end{array}
\right).$$ A feature of this triangle is that it is the \emph{reverse symmetrization} of the Riordan array
$$\left(\frac{1}{\sqrt{1-6x+x^2}}, \frac{1-x-\sqrt{1-6x+x^2}}{2}\right),$$ which is the Riordan array that begins
$$\left(
\begin{array}{ccccccc}
 1 & 0 & 0 & 0 & 0 & 0 & 0 \\
 3 & 1 & 0 & 0 & 0 & 0 & 0 \\
 13 & 5 & 1 & 0 & 0 & 0 & 0 \\
 63 & 25 & 7 & 1 & 0 & 0 & 0 \\
 321 & 129 & 41 & 9 & 1 & 0 & 0 \\
 1683 & 681 & 231 & 61 & 11 & 1 & 0 \\
 8989 & 3653 & 1289 & 377 & 85 & 13 & 1 \\
\end{array}
\right)$$
Denoting this matrix by $M$, the Delannoy square is given by
$$M^R+(M^R)^T-diag(1,3,13,63,\ldots),$$ were $M^R$ denotes the reversal of $M$. We subtract $diag(1,3,13,\ldots)$ as otherwise the diagonal would be duplicated.

The Delannoy square corresponds to the step set $\{(1,0),(0,1),(1,1)\}$. It we extend this step set to $\{(1,0),(0,1),(1,1),(2,2)\}$ we get the square matrix defined by
$$\frac{1}{1-x-y-xy-x^2y^2}.$$ This matrix begins
$$\left(
\begin{array}{ccccccc}
 1 & 1 & 1 & 1 & 1 & 1 & 1 \\
 1 & 3 & 5 & 7 & 9 & 11 & 13 \\
 1 & 5 & 14 & 27 & 44 & 65 & 90 \\
 1 & 7 & 27 & 71 & 147 & 263 & 427 \\
 1 & 9 & 44 & 147 & 379 & 816 & 1550 \\
 1 & 11 & 65 & 263 & 816 & 2082 & 4595 \\
 1 & 13 & 90 & 427 & 1550 & 4595 & 11651 \\
\end{array}
\right).$$
This is the symmetrization of the Riordan array
$$\left(\frac{1}{\sqrt{1-6x-x^2+2x^3+x^4}}, \frac{1-x-x^2-\sqrt{1-6x-x^2+2x^3+x^4}}{2}\right).$$ The first column of this array is \seqnum{A191649}.
When we read the above square matrix by diagonals, we get the Pascal-like matrix that has the generating function
$$\frac{1}{1-x-xy-x^2y-x^4y^2},$$ which counts paths from $(0,0)$ to $(n,k)$ with step set $\{(1,0),(1,1),(2,1),(4,2)\}$. This matrix begins
$$\left(
\begin{array}{ccccccc}
 1 & 0 & 0 & 0 & 0 & 0 & 0 \\
 1 & 1 & 0 & 0 & 0 & 0 & 0 \\
 1 & 3 & 1 & 0 & 0 & 0 & 0 \\
 1 & 5 & 5 & 1 & 0 & 0 & 0 \\
 1 & 7 & 14 & 7 & 1 & 0 & 0 \\
 1 & 9 & 27 & 27 & 9 & 1 & 0 \\
 1 &11 & 44 & 71 & 44 & 11 & 1 \\
\end{array}
\right).$$
This is not a Riordan array. The row sums
$$1, 2, 5, 12, 30, 74, 183, 452, 1117, 2760, 6820,\ldots,$$ have generating function
$$\frac{1}{1-2x-x^2-x^4},$$ and they count the left factors of the paths with step set  $\{(1,0),(1,1),(2,1),(4,2)\}$.
\end{example}

\begin{example} The (vertically) stretched Delannoy triangle $\left(\frac{1}{1-x}, \frac{x^2(1+x)}{1-x}\right)$ has generating function given by
$$\frac{\frac{1}{1-x}}{1-y\frac{x^2(1+x)}{1-x}}=\frac{1}{1-x-x^2y-x^3y}.$$
Thus this stretched triangle, which begins
$$\left(
\begin{array}{ccccccc}
 1 & 0 & 0 & 0 & 0 & 0 & 0 \\
 1 & 0 & 0 & 0 & 0 & 0 & 0 \\
 1 & 1 & 0 & 0 & 0 & 0 & 0 \\
 1 & 3 & 0 & 0 & 0 & 0 & 0 \\
 1 & 5 & 1 & 0 & 0 & 0 & 0 \\
 1 & 7 & 5 & 0 & 0 & 0 & 0 \\
 1 & 9 & 13 & 1 & 0 & 0 & 0 \\
\end{array}
\right),$$  counts the left factors for the step set $\{(1,0), (2,1), (3,1)\}$. The total number of left factors is given by the row sums of this matrix. We see that these are the tribonacci numbers \seqnum{A000073}
$$1, 1, 2, 4, 7, 13, 24, 44, 81, 149, 274,\ldots,$$ with generating function $\frac{1}{1-x-x^2-x^3}$.
The reversal of this stretched triangle, which begins,
$$\left(
\begin{array}{ccccccc}
 1 & 0 & 0 & 0 & 0 & 0 & 0 \\
 0 & 1 & 0 & 0 & 0 & 0 & 0 \\
 0 & 1 & 1 & 0 & 0 & 0 & 0 \\
 0 & 0 & 3 & 1 & 0 & 0 & 0 \\
 0 & 0 & 1 & 5 & 1 & 0 & 0 \\
 0 & 0 & 0 & 5 & 7 & 1 & 0 \\
 0 & 0 & 0 & 1 & 13 & 9 & 1 \\
\end{array}
\right),$$ has generating function
$$\frac{1}{1-xy-x^2y-x^3y^2}.$$
This triangle enumerates paths from $(0,0)$ to $(n,k)$ with steps in the set $\{(1,1), (2,1), (3,2)\}$.
\end{example}

\section{Triangulating rectified Riordan arrays}
In this section, we give a method for producing lower-triangular matrices from suitable square matrices.
For this, we let $\mathbf{B}$ denote the binomial matrix $\left(\binom{n}{k}\right)$. We furthermore let $\mathbf{B}(a)=\mathbf{B}_a$ denote the $a$-fold product of $\mathbf{B}$, that is, $\mathbf{B}^a$. This is the matrix
$\left(\binom{n}{k}a^{n-k}\right)$. As a Riordan array, it is given by
$$\left(\frac{1}{1-ax}, \frac{x}{1-ax}\right).$$
We have the following proposition.
\begin{proposition} Let $(g(x), f(x))$ be a Riordan array, and $\left(g(x), \frac{f(x)}{x}\right)$ be its rectification. If $f_2 \ne 0$, then the matrix
$$\left(g(x), \frac{f(x)}{x}\right) \cdot \left(\mathbf{B}_{f_1}^{-1}\right)^T$$ is a Riordan array.
\end{proposition}
\begin{proof}
The square matrix $\left(g(x), \frac{f(x)}{x}\right)$ has generating function
$$\frac{g(x)}{1-y\frac{f(x)}{x}}.$$
The matrix $\mathbf{B}_{f_1}^{-1}$ has generating function $$\left(\frac{1}{1+f_1y}, \frac{y}{1+f_1y}\right).$$
Multiplying on the right by the transpose indicates that we are operating on the $y$-variable. As $\mathbf{B}$ is a Riordan array, we can use the fundamental theorem of Riordan arrays to conclude that the generating function of the product is given by
$$\frac{1}{1+f_1 y} \frac{g(x)}{1-\frac{y}{1+f_1y}\frac{f(x)}{x}}=\frac{g(x)}{1+f_1y-y \frac{f(x)}{x}}.$$
Now $f(x)=f_1x+f_2x^2+\cdots$, so we obtain
$$\frac{g(x)}{1+f_1y-y(f_1+f_2x+\cdots)}=\frac{g(x)}{1-y(f_2 x + \cdots)}.$$
We have $f_2 x+\cdots = \frac{f(x)-f_1x}{x}$, and so we can conclude that the product $\left(g(x), \frac{f(x)}{x}\right) \cdot \left(\mathbf{B}_{f_1}^{-1}\right)^T$ is given by the Riordan array
$$\left(g(x), \frac{f(x)-f_1x}{x}\right)=(g(x), f_2x+\cdots).$$
\end{proof}
\begin{corollary} We have
$$\left(g(x), \frac{f(x)-f_1x}{x}\right)=(g(x),f(x))\cdot \left(1, \frac{x}{\bar{f}(x)}-f_1\right).$$
\end{corollary}
\begin{proof}
We have
\begin{align*}
(g,f)^{-1}\cdot \left(g(x), \frac{f(x)-f_1x}{x}\right)&=\left(\frac{1}{g(\bar{f})}, \bar{f}\right)\cdot \left(g(x), \frac{f(x)-f_1x}{x}\right)\\
&=\left(\frac{1}{g(\bar{f})} g(\bar{f}), \frac{f(\bar{f}(x))-f_1 \bar{f}(x)}{\bar{f}(x)}\right)\\
&=\left(1, \frac{x}{\bar{f}(x)}-f_1\right).\end{align*}
Multiplying both sides by $(g(x), f(x))$ now gives us
$$\left(g(x), \frac{f(x)-f_1x}{x}\right)=(g(x),f(x))\cdot \left(1, \frac{x}{\bar{f}(x)}-f_1\right).$$
\end{proof}
Note that this gives us
$$\left(g(x), \frac{f(x)-f_1x}{x}\right)=(g(x), f(x))\cdot (1, A(x)-f_1)$$ where $A(x)$ is the generating function of the $A$-sequence of $(g(x), f(x))$.
In the special case of a Bell matrix $(g(x), xg(x))$ this process has a special form. For simplicity, we assume that $g_0=1$. Rectifying $(g(x),xg(x))$ leads to $(g(x),g(x))$ with generating function
$$\frac{g(x)}{1-yg(x)}.$$ Applying the inverse binomial transform in $y$ gives us the generating function
$$\frac{1}{1+y} \frac{g(x)}{1-\frac{y}{1+y} g(x)}=\frac{g(x)}{1+y-yg(x)}=\mathcal{G}(g(x),g(x)-1).$$ This coincides with the notion of ``Riordan Square'' introduced by Peter Luschny \cite{Luschny}.  We can iterate this process, to get $\mathcal{G}\left(g(x), \frac{g(x)-1-x}{x}\right)$.
\begin{example} We consider the Riordan array $(g(x), f(x))=\left(\frac{1}{1-x-x^2}, \frac{x(1+x)}{1-x-x^2}\right)$. We have
$$f(x)=x+2x^2+3x^3+5x^4+\cdots,$$ so $f_2=2 \ne 0$. Also, $f_1=1$ so $\mathbf{B}_{f_1}=\mathbf{B}$. We conclude that the array $\left(g(x), \frac{f(x)}{x}\right) \cdot \left(\mathbf{B}_{f_1}^{-1}\right)^T$ is given by the Riordan array $$\left(\frac{1}{1-x-x^2}, \frac{x(2+x)}{1-x-x^2}\right).$$ This matrix begins as follows.
$$\left(
\begin{array}{ccccccc}
 1 & 0 & 0 & 0 & 0 & 0 & 0 \\
 1 & 2 & 0 & 0 & 0 & 0 & 0 \\
 2 & 5 & 4 & 0 & 0 & 0 & 0 \\
 3 & 12 & 16 & 8 & 0 & 0 & 0 \\
 5 & 25 & 49 & 44 & 16 & 0 & 0 \\
 8 & 50 & 127 & 166 & 112 & 32 & 0 \\
13 & 96 & 301 & 513 & 504 & 272 & 64 \\
\end{array}
\right).$$
The generating function of this array is
$$\frac{1}{1-x-x^2-2xy-x^2y}.$$ It therefore counts paths from $(0,0)$ to $(n,k)$ given steps $(1,0), (2,0), (2,1)$ and two kinds of step $(1,1)$. We denote such a step set by
$$\{(1,0), (2,0), 2*(1,1), (2,1)\}.$$ The row sums of this matrix count the left factors of such paths. We obtain the sequence \seqnum{A007482}
$$1, 3, 11, 39, 139, 495, 1763, 6279, 22363, 79647, 283667,\ldots$$ with generating function
$$\frac{1}{1-3x-2x^2}.$$  We recall that the original array
$\left(\frac{1}{1-x-x^2}, \frac{x(1+x)}{1-x-x^2}\right)$ corresponds to the step set $$\{(1,0), (2,0), (1,1), (2,1)\}.$$
We have seen that the left factors corresponding to this original array are counted by the sequence with generating function $$\gamma(x)=\frac{1}{1-2x-2x^2}.$$ We observe that the INVERT transform $\frac{g(x)}{1-xg(x)}$ of $g(x)$ is given by
$$\frac{\gamma(x)}{1-x \gamma(x)}=\frac{1}{1-3x-2x^2}.$$
\end{example}
\begin{example} We consider the Pascal-like matrix $\left(\frac{1}{1-x}, \frac{x(1+rx)}{1-x}\right)$. This has generating function
$$\frac{1}{1-x-xy-rx^2y},$$ corresponding to the step set
$$\{(1,0), (1,1), r*(2,1)\}$$ After rectifying and triangulating, we obtain the matrix
$$\left(\frac{1}{1-x}, \frac{(r+1)x}{1-x}\right).$$ This array has generating function
$$\frac{1}{1-x-(r+1)xy},$$ which corresponds to the step set
$$\{(1,0), (r+1)*(1,1)\}.$$ For $r=1$ (the Delannoy case), we obtain the matrix that begins
$$\left(
\begin{array}{ccccccc}
 1 & 0 & 0 & 0 & 0 & 0 & 0 \\
 1 & 2 & 0 & 0 & 0 & 0 & 0 \\
 1 & 4 & 4 & 0 & 0 & 0 & 0 \\
 1 & 6 & 12 & 8 & 0 & 0 & 0 \\
 1 & 8 & 24 & 32 & 16 & 0 & 0 \\
 1 & 10 & 40 & 80 & 80 & 32 & 0 \\
 1 & 12 & 60 & 160 & 240 & 192 & 64 \\
\end{array}
\right).$$ The row sums of this array are given by $3^n$, thus there are $3^n$ left factors (of length $n$) for the step set $\{(1,0), 2*(1,1)\}$. We note that the diagonal sums of this array are given by the positive Jacobsthal numbers \seqnum{A001045}. We can deduce from this that the positive Jacobsthal numbers count left factors for the step set $\{(1,0), 2*(2,1)\}$.
\end{example}

\section{Dyck paths, Motzkin paths and Catalan matrices}
A Dyck path is a path with step set $\{(1,1), (1,-1)\}$. The negative $y$-coordinate means that the analysis of these paths is different from the foregoing, involving the reversion of power series \cite{LI}, and, where Riordan arrays are concerned, the $A$-matrix formulation is often the most appropriate.
\begin{example}
The Riordan array $(c(x), xc(x))$, where $c(x)=\frac{1-\sqrt{1-4x}}{2x}$ is the generating function of the Catalan numbers $C_n=\frac{1}{n+1}\binom{2n}{n}$ \seqnum{A000108} begins
$$\left(
\begin{array}{ccccccc}
 1 & 0 & 0 & 0 & 0 & 0 & 0 \\
 1 & 1 & 0 & 0 & 0 & 0 & 0 \\
 2 & 2 & 1 & 0 & 0 & 0 & 0 \\
 5 & 5 & 3 & 1 & 0 & 0 & 0 \\
 14 & 14 & 9 & 4 & 1 & 0 & 0 \\
 42 & 42 & 28 & 14 & 5 & 1 & 0 \\
 132 & 132 & 90 & 48 & 20 & 6 & 1 \\
\end{array}
\right).$$ This triangle \seqnum{A033184} counts the number of Dyck paths of semi-length $n$ and having first peak at height $k$. It also counts the number of Dyck paths of semi-length $n$ and having $k$ returns to the $x$-axis. The triangulation of the corresponding square matrix is the Riordan array $(c(x), xc(x)^2)$ \seqnum{A039599}, which begins
$$\left(
\begin{array}{ccccccc}
 1 & 0 & 0 & 0 & 0 & 0 & 0 \\
 1 & 1 & 0 & 0 & 0 & 0 & 0 \\
 2 & 3 & 1 & 0 & 0 & 0 & 0 \\
 5 & 9 & 5 & 1 & 0 & 0 & 0 \\
 14 & 28 & 20 & 7 & 1 & 0 & 0 \\
 42 & 90 & 75 & 35 & 9 & 1 & 0 \\
 132 & 297 & 275 & 154 & 54 & 11 & 1 \\
\end{array}
\right).$$ This counts the number of grand Dyck paths of semi-length $n$ which have $k$ \emph{downward} returns to the $x$-axis. It also counts the number of Motzkin paths from $(0,0)$ to $(n,k)$, when for levels above the $x$-axis, the level steps $(1,0)$ can be of two colors.

We can iterate this process once more, to get the Riordan array $(c(x), c(x)^2-1)$ which begins
$$\left(
\begin{array}{ccccccc}
 1 & 0 & 0 & 0 & 0 & 0 & 0 \\
 1 & 2 & 0 & 0 & 0 & 0 & 0 \\
 2 & 7 & 4 & 0 & 0 & 0 & 0 \\
 5 & 23 & 24 & 8 & 0 & 0 & 0 \\
 14 & 76 & 109 & 68 & 16 & 0 & 0 \\
 42 & 255 & 449 & 394 & 176 & 32 & 0 \\
 132 & 869 & 1770 & 1947 & 1240 & 432 & 64 \\
\end{array}
\right).$$
This array has row sums
$$1, 3, 13, 60, 283, 1348, 6454, 30992, 149091, 718044, 3460818,\ldots,$$ with generating function
$$\frac{c(x)}{2-c(x)^2}=\frac{1-2x+(1+2x)\sqrt{1-4x^2}}{2(1-4x-4x^2)}.$$
An interesting feature of this sequence is that its Hankel transform \cite{Layman} is the sequence \seqnum{A001353}
$$1, 4, 15, 56, 209, 780, 2911, 10864, 40545, 151316, 564719,\ldots,$$ with generating function
$\frac{1}{1-4x+x^2}$.

\end{example}
\begin{example}
When considering the triangulation of a square matrix derived from a Riordan array $(g(x), f(x))$, we stipulated that $f_2 \ne 0$. In the following example, we look at a case where $f_2=0$.
Thus we consider the Riordan array
$$\left(\frac{1}{1+x^2}, \frac{x}{1+x^2}\right)^{-1}=\left(c(x^2),xc(x^2)\right)$$ which begins
$$\left(
\begin{array}{ccccccc}
 1 & 0 & 0 & 0 & 0 & 0 & 0 \\
 0 & 1 & 0 & 0 & 0 & 0 & 0 \\
 1 & 0 & 1 & 0 & 0 & 0 & 0 \\
 0 & 2 & 0 & 1 & 0 & 0 & 0 \\
 2 & 0 & 3 & 0 & 1 & 0 & 0 \\
 0 & 5 & 0 & 4 & 0 & 1 & 0 \\
 5 & 0 & 9 & 0 & 5 & 0 & 1 \\
\end{array}
\right).$$ This matrix counts Dyck paths from $(0,0)$ to $(n,k)$. Applying $\left(\mathbf{B}^{-1}\right)^T$ to the corresponding square matrix $(c(x^2), c(x^2))$, we obtain, not a lower-triangular matrix, but the matrix that begins
$$\left(
\begin{array}{ccccccc}
 1 & 0 & 0 & 0 & 0 & 0 & 0 \\
 0 & 0 & 0 & 0 & 0 & 0 & 0 \\
 1 & 1 & 0 & 0 & 0 & 0 & 0 \\
 0 & 0 & 0 & 0 & 0 & 0 & 0 \\
 2 & 3 & 1 & 0 & 0 & 0 & 0 \\
 0 & 0 & 0 & 0 & 0 & 0 & 0 \\
 5 & 9 & 5 & 1 & 0 & 0 & 0 \\
\end{array}
\right),$$ with generating function
$$\frac{c(x^2)}{1-y(c(x^2)-1)}.$$
This is a vertically aerated version of the Catalan matrix $(c(x), xc(x)^2)=(c(x),c(x)-1)$.
\end{example}
\begin{example} \textbf{Motzkin paths} A Motzkin path is a path from $(0,0)$ to $(n,0)$ using steps $(1,0), (1,1)$ and $(1,-1)$. The number of such numbers is given by the Motzkin numbers
$$1, 1, 2, 4, 9, 21, 51, 127, 323, 835, 2188, 5798, 15511, 41835, \ldots,$$ with generating function
$$M(x)=\frac{1-x-\sqrt{1-2x-3x^2}}{2x^2}.$$ The number of such paths from $(0,0)$ to $(n,k)$ is given by the Bell matrix $(M(x), xM(x))=\left(\frac{1}{1+x+x^2}, \frac{x}{1+x+x^2}\right)^{-1}$ \seqnum{A064189}. As a Bell matrix, we find that the triangulation of the corresponding square matrix is given by $(M(x), M(x)-1)$ \seqnum{A321621}. We note that we have
$$(M(x), M(x)-1)^{-1}=\left(\frac{1}{1+x}, \frac{xc(-x)}{1+x}\right).$$ The generating function of the matrix $(M(x), M(x)-1)$ may be represented as a Jacobi continued fraction as
$$\cfrac{1}{1-(y+1)x-\cfrac{(y+1)x^2}{1-x-\cfrac{x^2}{1-x-\cfrac{x^2}{1-x-\cdots}}}}.$$ We have
$$(M(x), M(x)-1) = (M(x), xM(x))\cdot (1, x(1+x)).$$ The row sums of the matrix $(1,x(1+x))$ are the positive Fibonacci numbers, and so the row sums \seqnum{A111961} of $(M(x), M(x)-1)$, which begin
$$1, 2, 6, 18, 56, 176, 558, 1778, 5686, 18230, 58558,\ldots,$$
are the image of the positive Fibonacci numbers by the Motzkin matrix $(M(x), xM(x))$.
\end{example}
Our final example of this section starts with a lattice path Riordan array that is not a Bell matrix.
\begin{example} The Riordan array
$$\tilde{M}=\left(\frac{1+x}{1+3x+x^2}, \frac{x}{1+3x+x^2}\right)^{-1}=\left(\frac{1-x-\sqrt{1-6x+5x^2}}{2x(1-x)}, \frac{1-3x-\sqrt{1-6x+5x^2}}{2x}\right)$$ counts Motzkin paths from $(0,0)$ to $(n,k)$, where the level paths $(1,0)$ for levels above the ground level have $3$ colors. This matrix begins
$$\left(
\begin{array}{ccccccc}
 1 & 0 & 0 & 0 & 0 & 0 & 0 \\
 2 & 1 & 0 & 0 & 0 & 0 & 0 \\
 5 & 5 & 1 & 0 & 0 & 0 & 0 \\
 15 & 21 & 8 & 1 & 0 & 0 & 0 \\
 51 & 86 & 46 &11 & 1 & 0 & 0 \\
188 & 355 & 235 & 80 & 14 & 1 & 0 \\
731 & 1488 & 1140 & 489 & 123 & 17 & 1 \\
\end{array}
\right).$$ Then the triangulation of the rectified matrix is given by the Riordan array
$$T=\left(\frac{1-x-\sqrt{1-6x+5x^2}}{2x(1-x)}, \frac{1-3x-2x^2-\sqrt{1-6x+5x^2}}{2x^2}\right).$$ This array begins
$$\left(
\begin{array}{ccccccc}
 1 & 0 & 0 & 0 & 0 & 0 & 0 \\
 2 & 3 & 0 & 0 & 0 & 0 & 0 \\
 5 & 16 & 9 & 0 & 0 & 0 & 0 \\
 15 & 71 & 78 & 27 & 0 & 0 & 0 \\
 51 & 304 & 481 & 324 & 81 & 0 & 0 \\
188 & 1300 & 2609 & 2547 & 1242 & 243 & 0 \\
731 & 5604 & 13317 & 16678 & 11853 & 4536 & 729 \\
\end{array}
\right).$$ It has row sums
$$1, 5, 30, 191, 1241, 8129, 53448, 352097, 2321962, 15322025, 101143706,\ldots$$
whose generating function is given by
$$\frac{2x}{(1+2)\sqrt{1-6x+5x^2}-1+3x-2x^2}.$$
The Hankel transform of this sequence is
$$1, 5, 24, 115, 551, 2640, 12649, 60605, 290376, 1391275,\ldots,$$
with generating function
$$\frac{1}{1-5x+x^2}.$$
We have $$T=\tilde{M} \cdot (1, x(3+x)),$$ and hence the row sums of $T$ are the transform by $\tilde{M}$ of the sequence \seqnum{A006190}
$$1, 3, 10, 33, 109, 360, 1189, 3927, 12970, 42837, 141481,\ldots$$ with generating function
$$\frac{1}{1-3x-x^2}.$$
\end{example}

\section{Almost Riordan arrays}
An \emph{almost Riordan array of first order} \cite{almost} is a lower-triangular matrix $(a_{n,k})_{n,k \ge 0}$ such that the matrix $(a_{n,k})_{n,k \ge 1}$ is a Riordan array $(g(x), f(x))$.
For such a matrix, if the generating function of the first column $a_{n,0}$ is $a(x)$, we use the notation $(a(x), g(x), f(x))$ or $(a(x); g(x), f(x))$. The generating function of such a matrix is given by
$$a(x)+xy\frac{g(x)}{1-yf(x)}.$$
We give two examples of such matrices.
\begin{example}  We consider the set of lattice paths from $(0,0)$ to $(n,k)$ such at level $0$, the step set is $\{(1,0), (1,1)\}$, at level $1$ the step set is
$\{(1,0), (1,1),(2,1)\}$, and thereafter the step set is $\{(1,1), (2,0),(2,1)\}$. We find that the corresponding matrix, which begins
$$\left(
\begin{array}{ccccccc}
 1 & 0 & 0 & 0 & 0 & 0 & 0 \\
 1 & 1 & 0 & 0 & 0 & 0 & 0 \\
 1 & 3 & 1 & 0 & 0 & 0 & 0 \\
 1 & 5 & 4 & 1 & 0 & 0 & 0 \\
 1 & 7 & 9 & 5 & 1 & 0 & 0 \\
 1 & 9 & 16 & 14 & 6 & 1 & 0 \\
 1 & 11 & 25 & 30 & 20 & 7 & 1 \\
\end{array}
\right)$$ is the almost Riordan array
$$\left(\frac{1}{1-x}, \frac{1+x}{(1-x)^2}, \frac{x}{1-x}\right)$$ with row sums
$$1, 2, 5, 11, 23, 47, 95, 191, 383, 767, 1535,\ldots,$$ whose generating function is
$$\frac{1-x+x^2}{(1-x)(1-2x)}.$$ This sequence thus counts left factors of the paths of this example. The generating function of this triangle is given by
$$\frac{1-x+x^2y}{1-2x+x^2-xy+x^2y}.$$
\end{example}
\begin{example} In this example, we initially look at the array associated to the left factors of paths whose step set is $\{(2,0),(1,1)\}$ at level $0$, and subsequently has step set $\{(1,0),(1,1)\}$.
This is the array $$\left(\frac{1}{1-x^2}, \frac{x}{1-x}\right),$$ whose generating function is
$$\frac{1}{1-x^2-xy-x^2y}.$$ But this is the left factor array of paths with step set $\{(2,0),(1,1),(2,1)\}$ at all levels. Thus two (or more) specifications may lead to the same path set. We note that the row sums of $\left(\frac{1}{1-x^2}, \frac{x}{1-x}\right)$ are given by the Jacobsthal numbers.

Going back to the original specification, we now confine the step set $\{(1,0),(1,1)\}$ to level $1$, and use the step set $\{(1,0), (1,1), (2,1)\}$ subsequently. We arrive at a matrix that begins
$$\left(
\begin{array}{ccccccc}
 1 & 0 & 0 & 0 & 0 & 0 & 0 \\
 0 & 1 & 0 & 0 & 0 & 0 & 0 \\
 1 & 1 & 1 & 0 & 0 & 0 & 0 \\
 0 & 2 & 3 & 1 & 0 & 0 & 0 \\
 1 & 2 & 6 & 5 & 1 & 0 & 0 \\
 0 & 3 & 10 & 14 & 7 & 1 & 0 \\
 1 & 3 & 15 & 30 & 26 & 9 & 1 \\
\end{array}
\right).$$
This is the almost-Riordan array
$$\left(\frac{1}{1-x^2}, \frac{1+x}{(1-x^2)^2}, \frac{x(1+x)}{1-x}\right).$$
Its row sums \seqnum{A113225}
$$1, 1, 3, 6, 15, 35, 85, 204, 493, 1189, 2871,\ldots$$ have generating function $\frac{1-x-x^2}{(1-x^2)(1-2x-x^2)}$. The generating function of the left factor array is
$$\frac{1-x-x^2y}{1-x-x^2+x^3-x^2y+x^3y+x^4y}.$$
\end{example}
\begin{example} We consider the set of lattice paths from $(0,0)$ to $(n,k)$ such that at level $0$, the step set is $\{(1,-1), (1,1)\}$ (Dyck steps), at level $1$ the step set is
$\{(1,-1), (1,0),(1,1)\}$ (Motzkin steps), and thereafter the step set is $\{(1,-1), (2,0),(1,1)\}$ (Schroeder steps). We find that the corresponding matrix, which begins
$$\left(
\begin{array}{ccccccc}
 1 & 0 & 0 & 0 & 0 & 0 & 0 \\
 0 & 1 & 0 & 0 & 0 & 0 & 0 \\
 1 & 1 & 1 & 0 & 0 & 0 & 0 \\
 1 & 3 & 1 & 1 & 0 & 0 & 0 \\
 3 & 5 & 1 & 1 & 1 & 0 & 0 \\
 5 & 13 & 7 & 7 & 1 & 1 & 0 \\
13 & 25 & 25 & 9 & 9 & 1 & 1 \\
\end{array}
\right)$$
is the almost Riordan array
$$\left(1-\frac{x(1-2x-x^2-\sqrt{1-6x^2+x^4}}{2(1-2x-x^2)}, -\frac{1-2x-x^2-\sqrt{1-6x^2+x^4}}{2x(1-2x-x^2)}, \frac{1-x^2-\sqrt{1-6x^2+x^4}}{2x}\right).$$
It is interesting to note that the row sums of the Riordan array $(g(x), f(x))$ are the Pell numbers \seqnum{A000129} with generating function $\frac{1}{1-2x-x^2}$. The first column of $(g(x), f(x))$ is \seqnum{A026003}, which counts left-factors of Schroeder paths. The left factors of the paths of this example are counted by the sequence
$$1, 1, 3, 6, 15, 34, 83, 194, 471, 1114, 2699,\ldots$$
with generating function
$$\frac{2-3x+x^3+x \sqrt{1-6x^2+x^4}}{2(1-2x-x^2)}.$$
We note that
$$1-6x^2+x^4=(1-2x-x^2)(1+2x-x^2).$$
\end{example}

\section{Downward steps}
When downward steps are present, the nature of the generating functions changes. We give examples to make this explicit.
\begin{example} We consider Dyck paths, with steps $\{(1,1), (1,-1)\}$. The left factor matrix in this case, that enumerates Dyck paths from $(0,0)$ to $(n,k)$, is given by
 $$(c(x^2), xc(x^2)),$$ where $c(x)=\frac{1-\sqrt{1-4x}}{2x}$. The generating function of this Riordan array can be written
 $$\frac{c(x^2)}{1-yxc(x^2)}=\left(1-\frac{xc(x^2)}{y}\right) \frac{1}{1-xy-x/y}.$$
 Thus the rational expression $\frac{1}{1-x^1y^1-x^1y^{-1}}$ expressing the step set $\{(1,1), (1,-1)\}$ now appears as a factor in the generating function.
 \end{example}
 \begin{example} We consider Motzkin paths, with steps $\{(1,1),(1,0),(1,-1)\}$. The left factor matrix in this case, that enumerates Motzkin paths from $(0,0)$ to $(n,k)$, is given by the Motzkin matrix
 $(M(x), x M(x))$, where $M(x)=\frac{1-x-\sqrt{1-2x-3x^2}}{2x^2}$ is the generating function of the Motzkin numbers. The generating function of the Motzkin matrix can be written
 $$ \frac{M(x)}{1-yxM(x)}=\left(1-\frac{xM(x)}{y}\right)\frac{1}{1-x-xy-x/y}.$$
 Thus again we have that the rational quotient  $\frac{1}{1-x^1y^0-x^1y^1-x^1y^{-1}}$ that exhibits the step set $\{(1,1),(1,0),(1,-1)\}$ appears as a factor in the generating function.
 \end{example}
 \begin{example} The Schroeder paths, with step set $\{(2,0),(1,1),(1,-1)\}$, have their left factor matrix given by
 $(S(x^2), xS(x^2))$, where $S(x)=\frac{1-x-\sqrt{1-6x+x^2}}{2x}$. The generating function of this array can be written as
 $$\left(1-\frac{xS(x^2)}{y}\right)\frac{1}{1-x^2y^0-x^1y^1-x^1y^{-1}},$$  exhibiting the step set.

 If we start with the Schroeder matrix $(S(x), xS(x))$, we find that its generating function is given by
 $$\left(1-\frac{S(x)}{y}\right)\frac{1}{1-x-xy-1/y}.$$ We deduce that the Schroeder matrix $(S(x), xS(x))$ \seqnum{A080247}, which begins
 $$\left(
\begin{array}{ccccccc}
 1 & 0 & 0 & 0 & 0 & 0 & 0 \\
 2 & 1 & 0 & 0 & 0 & 0 & 0 \\
 6 & 4 & 1 & 0 & 0 & 0 & 0 \\
 22 & 16 & 6 & 1 & 0 & 0 & 0 \\
 90 & 68 & 30 & 8 & 1 & 0 & 0 \\
 394 & 304 & 146 & 48 & 10 & 1 & 0 \\
 1806 & 1412 & 714 & 264 & 70 & 12 & 1 \\
\end{array}
\right),$$
 counts paths from $(0,0)$ to $(n,k)$ with step set $\{(1,1), (1,0), (0,-1)\}$.
 \end{example}
 \begin{example} The Catalan $(c(x), xc(x))$ has a generating function that can be expressed as
 $$\frac{c(x)}{1-yxc(x)}=\left(1-\frac{c(x)}{y}\right)\frac{1}{1-xy-y^{-1}}.$$
 This indicates that the corresponding matrix $(t_{n,k})$ satisfies the recurrence
 $$t_{n,k}=t_{n-1,k-1}+t_{n,k+1},$$ subject to $t_{0,0}=1$ and $t_{n,k}=0$ if $n<0$ or if $k<0$ or if $k>n$. The rational quotient
 $\frac{1}{1-xy-1/y}$ in the generating function exhibits the step set $\{(1,1), (0,-1)\}$. This Catalan matrix then counts paths from $(0,0)$ to $(n,k)$ for this step set. See Figure \ref{Cat}.

\begin{figure}
\begin{center}
\begin{tikzpicture}
\draw(0,0)--(6,0);
\draw(0,0)--(6,6);
\draw(1,0)--(1,1.2);
\draw(2,0)--(2,2.2);
\draw(3,0)--(3,3.2);
\draw(4,0)--(4,4.2);
\draw(5,0)--(5,5.2);
\draw(6,0)--(6,6.2);

\draw(1,1)--(6.2,1);
\draw(2,2)--(6.2,2);
\draw(3,3)--(6.2,3);
\draw(4,4)--(6.2,4);
\draw(5,5)--(6.2,5);

\draw(1,0)--(6,5);
\draw(2,0)--(6,4);
\draw(3,0)--(6,3);
\draw(4,0)--(6,2);
\draw(5,0)--(6,1);

\draw[red, thick, postaction={decorate},
      decoration={markings, mark=at position 0.5 with {\arrow{>}}}]
      (1,1) -- (1,0);

\draw[red, thick, postaction={decorate},
      decoration={markings, mark=at position 0.5 with {\arrow{>}}}]
      (2,2) -- (2,1);

\draw[red, thick, postaction={decorate},
      decoration={markings, mark=at position 0.5 with {\arrow{>}}}]
      (2,1) -- (2,0);

\draw[red, thick, postaction={decorate},
      decoration={markings, mark=at position 0.5 with {\arrow{>}}}]
      (3,3) -- (3,2);

\draw[red, thick, postaction={decorate},
      decoration={markings, mark=at position 0.5 with {\arrow{>}}}]
      (3,2) -- (3,1);

\draw[red, thick, postaction={decorate},
      decoration={markings, mark=at position 0.5 with {\arrow{>}}}]
      (3,1) -- (3,0);

\draw[blue, thick, postaction={decorate},
      decoration={markings, mark=at position 0.5 with {\arrow{>}}}]
      (0,0) -- (1,1);

\draw[blue, thick, postaction={decorate},
      decoration={markings, mark=at position 0.5 with {\arrow{>}}}]
      (1,1) -- (2,2);

\draw[blue, thick, postaction={decorate},
      decoration={markings, mark=at position 0.5 with {\arrow{>}}}]
      (1,0) -- (2,1);

\draw[blue, thick, postaction={decorate},
      decoration={markings, mark=at position 0.5 with {\arrow{>}}}]
      (2,1) -- (3,2);

\draw[blue, thick, postaction={decorate},
      decoration={markings, mark=at position 0.5 with {\arrow{>}}}]
      (2,0) -- (3,1);

\draw[blue, thick, postaction={decorate},
      decoration={markings, mark=at position 0.5 with {\arrow{>}}}]
      (2,2) -- (3,3);

\node at (0,-.3) {1};
\node at (1,1.3) {1};
\node at (2,2.3) {1};
\node at (3,3.3) {1};
\node at (4,4.3) {1};
\node at (5,5.3) {1};
\node at (6,6.3) {1};

\node at (1,-.3) {1};
\node at (2,-.3) {2};
\node at (3,-.3) {5};
\node at (4,-.3) {14};
\node at (5,-.3) {42};

\node at (0,-.3) {1};
\node at (1,1.3) {1};
\node at (2,2.3) {1};
\node at (3,3.3) {1};
\node at (4,4.3) {1};
\node at (5,5.3) {1};
\node at (6,6.3) {1};

\node at (2,1.3) {2};
\node at (3,1.3) {5};

\node at (3,2.3) {3};

\node at (4,1.3) {14};
\node at (4,2.3) {9};
\node at (4,3.3) {4};

\node at (5,1.3) {42};
\node at (5,2.3) {28};
\node at (5,3.3) {14};
\node at (5,4.3) {5};
\end{tikzpicture}
\end{center}
\caption{The Catalan matrix $(c(x),xc(x))$ as a path matrix}\label{Cat}
\end{figure}
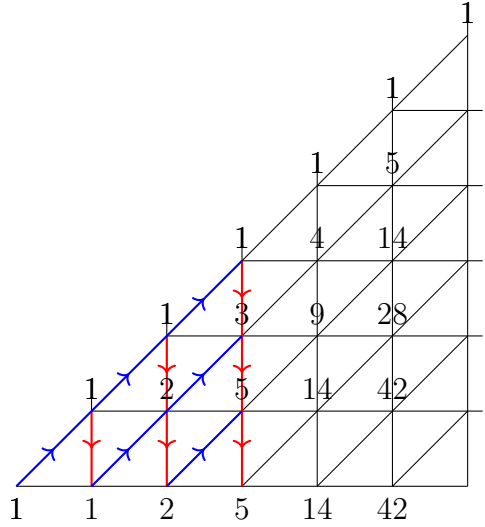
 \end{example}
 We end this section with an example which begins with a generating function, and seeks to reverse the foregoing procedure. Thus we consider the Riordan array
 $$(g(x), xg(x))=\left(\frac{1-x-\sqrt{1-6x+5x^2}}{2x}, \frac{1-x-\sqrt{1-6x+5x^2}}{2}\right).$$
 The generating function of this array can be expressed as
 $$\left(1-x-\frac{g(x)}{y}\right)\frac{1}{1-x-xy+x^2y-1/y}.$$ This leads us to consider the recurrence
 $$t_{n,k}=t_{n-1,k}+t_{n-1,k}-t_{n-2,k-1}+t_{n,k+1},$$ subject to $t_{0,0}=1$ and $t_{n,k}=0$ if $n<0$ or if $k<0$ or if $k>n$.
 We find that the matrix $(t_{n,k})$ is the matrix
 $$(\tilde{g}(x), xg(x))=\left(\frac{1-x-\sqrt{1-6x+5x^2}}{2x(1-x)}, \frac{1-x-\sqrt{1-6x+5x^2}}{2}\right),$$ which begins
$$\left(
\begin{array}{ccccccc}
 1 & 0 & 0 & 0 & 0 & 0 & 0 \\
 2 & 1 & 0 & 0 & 0 & 0 & 0 \\
 5 & 3 & 1 & 0 & 0 & 0 & 0 \\
15 & 10 & 4 & 1 & 0 & 0 & 0 \\
51 & 36 & 16 & 5 & 1 & 0 & 0 \\
188 & 137 & 65 & 23 & 6 & 1 & 0 \\
731 & 543 & 269 & 103 & 31 & 7 & 1 \\
\end{array}
\right).$$
The generating function of the Riordan array $(\tilde{g}(x), xg(x))$ can be expressed as
$$\left(1-\frac{\tilde{g}(x)}{y}\right)\frac{1}{1-x-xy+x^2y-1/y}.$$
The above matrix then corresponds to paths from $(0,0)$ to $(n,k)$ for the special step set $\{(1,0), (1,1), (-1)*(2,1),(0,-1)\}$.

\section{A result on equinumerous lattice paths}
In this section we describe a result that relate colored Motzkin and Schroeder paths to a third family of lattice paths. For this, we consider the generating function
$$g(x)=g(x;r,s)=\frac{1-rx-\sqrt{1-2(r+2s)x+r^2x^2}}{2sx},$$ and the Bell matrix $(g(x), xg(x))$.
The generating function $g(x)$ can be described by both a Jacobi continued fraction and a Thron continued fraction. We have
$$g(x)=\cfrac{1}{1-rx-\cfrac{sx}{1-rx-\cfrac{sx}{1-rx-\cdots}}},$$ showing that $g(x)$ enumerates $(r,s)$-Schroeder lattice paths \cite{Petreolle}. We also have
$$g(x)=\cfrac{1}{1-(r+s)x-\cfrac{s(r+s)x^2}{1-(r+2s)x-\cfrac{s(r+s)x^2}{1-(r+2s)x-\cfrac{s(r+s)x^2}{1-(r+2s)x-\cdots}}}}.$$ This shows that $g(x)$ enumerates generalized Motzkin paths with
$r+s$ types of level steps at level $0$, $r+2s$ level steps at subsequent levels, and $s(r+s)$ types of down steps.
We now consider the bivariate generating function $\frac{g(x;r,s)}{1-yxg(x;r,s)}$ of the Bell matrix $(g(x), xg(x))$. This can be expressed as
$$\frac{g(x;r,s)}{1-yxg(x;r,s)}=\left(1-\frac{sg(x)}{y}\right)\frac{1}{1-rx-xy-s/y}.$$
Thus the Bell matrix $(g(x), xg(x))$ counts lattice paths from $(0,0)$ to $(n,k)$ with $r$ types of $(1,0)$ step, a $(1,1)$ step, and $s$ types of $(0,-1)$ step. In particular, the number of such paths from $(0,0)$ to $(n,0)$ is equinumerous with the number of Motzkin and Schroeder paths above.
\begin{example} For $r=0$ and $s=1$, we obtain $g(x)=c(x)$, and then the Bell matrix $(g(x), xg(x))$ is the Catalan matrix $(c(x), xc(x))$.
When $r=1$ and $s=1$, we obtain $g(x)=S(x)$ and $(g(x),xg(x))=(S(x),xS(x))$.
\end{example}

\section{Using the $A$-matrix}
We can approach the above considerations from the point of view of the $A$-matrix formulation of the Riordan arrays. The lattice paths above satisfy the recurrence
$$t_{n,k}=t_{n-1,k-1}+r t_{n-1,k}+s t_{n,k+1},$$ subject to $t_{n,k}=0$ if $n<0$ or $k<0$ or $k>n$, and $t_{0,0}=1$. Using the $A$-matrix approach, this gives us the equation
$$\frac{u}{x}=1+ru+\frac{su^2}{x},$$ whose solution is $u(x)=x g(x;r,s)$.

In general, we have the following schema.
$$\left(
\begin{array}{ccccccc}
\vdots & \cdots & \cdots & \cdots & \cdots & \cdots & \cdots \\
x^2 & \cdots & t_{n-3,k-1} & t_{n-3,k} & t_{n-3,k+1} & t_{n-3,k+2} & \cdots \\
x & \cdots & t_{n-2,k-1} & t_{n-2,k} & t_{n-2,k+1} & t_{n-2,k+2} & \cdots \\
1 & \cdots & t_{n-1,k-1} & t_{n-1,k} & t_{n-1,k+1} & t_{n-1,k+2} & \cdots \\
1/x & \cdots & t_{n,k-1} & t_{n,k} & t_{n,k+1} & t_{n,k+2} & \cdots \\
1/x^2 & \cdots &  &  &  & t_{n+1,k+2} & \cdots \\
 &  &  &  &  &  &  \\
\cdots & \cdots & 1 & u & u^2 & u^3 & \cdots \\
\end{array}
\right).$$
Thus if, for instance, we have
$$t_{n,k}=t_{n-1,k-1}+ \alpha t_{n-2,k}+ \beta t_{n-2,k+1}+ \gamma t_{n,k+1},$$ with standard initial conditions, then we have
$$\frac{u}{x}=1+\alpha xu+\beta xu^2+ \gamma \frac{u^2}{x},$$ to give
$$u=f(x)=\frac{1-\alpha x^2-\sqrt{1-4 \gamma x - 2\alpha x^2-4 \beta x^3+\alpha^2 x^4}}{2(\gamma+\beta x^2)}.$$
We note that in this case we have
$$g(x)=\frac{f(x)}{x}=\frac{1}{1-\alpha x^2}c\left(\frac{x(\gamma+\beta x^2)}{(1-\alpha x^2)^2}\right).$$
Then the Hankel transform of the expansion of $g(x)$ is a
$$(\alpha \gamma+ \beta+2\gamma^3)^2, -\alpha^3 \gamma^2-\alpha^2\gamma(2\beta+3\gamma^3)-\alpha(\beta^2+6\beta \gamma^3+4 \gamma^6)-\gamma(2\beta^2+6\beta \gamma^3+3\gamma^6)$$
Somos-$4$ sequence.

The Riordan array $(g(x), xg(x))$ then counts lattice paths from $(0,0)$ to $(n,k)$ with steps $(1,1)$, $\alpha$ types of step $(2,0)$, $\beta$ types of step $(2,-1)$ and $\gamma$ types of step $(0,-1)$.

\section{Two matrices}
We consider the two equations
$$\frac{u}{x}=1+u+\frac{u^2}{x},$$ with solution
$$u(x)=xg(x)=xS(x),$$
and the equation
$$\frac{u}{x}=1+xu+\frac{u^2}{x},$$ with solution
$$u(x)=xg(x)=\frac{1-x-\sqrt{1-4x-2x^2+x^4}}{2}.$$
These equations correspond to Riordan arrays $(g(x), xg(x))$ that are the left factor arrays, respectively, for paths with step set $\{(1,1), (1,0), (0,-1)\}$, respectively, $\{(1,1), (2,0), (0,-1)\}$.
The expansions of $g(x)$ correspond, respectively, to the row sums and the diagonal sums of the matrix \seqnum{A060693} with generating function
$$\frac{1-xy-\sqrt{1-4x-2xy+x^2y^2}}{2x},$$ which begins
$$\left(
\begin{array}{ccccccc}
1 & 0 & 0 & 0 & 0 & 0 & 0 \\
1 & 1 & 0 & 0 & 0 & 0 & 0 \\
2 & 3 & 1 & 0 & 0 & 0 & 0 \\
5 & 10 & 6 & 1 & 0 & 0 & 0 \\
14 & 35 & 30 & 10 & 1 & 0 & 0 \\
42 & 126 & 140 & 70 & 15 & 1 & 0 \\
132 & 462 & 630 & 420 & 140 & 21 & 1 \\
\end{array}
\right).$$ The general term of this matrix is
$$\frac{\binom{2n-k}{k} \binom{2n-2k}{n-k}}{n-k+1}.$$ The row sums are the Schroeder numbers \seqnum{A006318}
$$1, 2, 6, 22, 90, 394, 1806, 8558,\ldots,$$ and the diagonal sums are the sequence \seqnum{A143330}
$$1, 1, 3, 8, 25, 83, 289, 1041,\ldots.$$ We have already seen that the first column numbers (the Catalan numbers) correspond to lattice paths with the step set $\{1,1),(0,-1)\}$.
We note that the above matrix is equal to the Narayana matrix \seqnum{A090181} multiplied by the binomial matrix (on the right).

We now consider the two equations
$$\frac{u}{x}=1+u+\frac{u^3}{x^2},$$ with solution
$$u(x)=xg(x)=\frac{2}{\sqrt{3}}\sqrt{x(1-x)}\sin\left(\frac{1}{3}\sin^{-1}\left(\frac{3\sqrt{3}}{2} \frac{\sqrt{x(1-x)}}{(1-x)^2}\right)\right),$$
and the equation
$$\frac{u}{x}=1+xu+\frac{u^3}{x^2},$$ with solution
$$u(x)=xg(x)=\frac{2}{\sqrt{3}}\sqrt{x(1-x^2)}\sin\left(\frac{1}{3}\sin^{-1}\left(\frac{3\sqrt{3}}{2} \frac{\sqrt{x(1-x^2)}}{(1-x^2)^2}\right)\right).$$
We obtain the sequences (expansion of $g(x)$), respectively, \seqnum{A346626}
$$1, 2, 8, 44, 280, 1936, 14128, 107088, 834912,\ldots,$$ and
\seqnum{A364474}
$$1, 1, 4, 16, 77, 403, 2228, 12800, 75653, 457022, 2809266,\ldots.$$
These equations correspond to Riordan arrays $(g(x), xg(x))$ that are the left factor arrays, respectively, for paths with step set $\{(1,1), (1,0), (-1,-2)\}$, respectively, $\{(1,1), (2,0), (-1,-2)\}$.
The corresponding recurrences in the case are
$$t_{n,k}=t_{n-1,k-1}+t_{n-1,k}+t_{n+1,k+2},$$ respectively
$$t_{n,k}=t_{n-1,k-1}+t_{n-2,k}+t_{n+1,k+2}.$$
These sequences then appear as the row sums and the diagonal sums of the array $\mathbf{T}$ with generating function
$$\frac{2}{\sqrt{3}}\sqrt{x(1-xy)}\sin\left(\frac{1}{3}\sin^{-1}\left(\frac{3\sqrt{3}}{2} \frac{\sqrt{x(1-xy)}}{(1-xy)^2}\right)\right).$$
This matrix begins
$$\left(
\begin{array}{ccccccc}
 1 & 0 & 0 & 0 & 0 & 0 & 0 \\
 1 & 1 & 0 & 0 & 0 & 0 & 0 \\
 3& 4 & 1 & 0 & 0 & 0 & 0 \\
12 & 21 & 10 & 1 & 0 & 0 & 0 \\
55 & 120 & 84 & 20 & 1 & 0 & 0 \\
273 &715 & 660 & 252 & 35 & 1 & 0 \\
1428 & 4368 & 5005 & 2640 & 630 & 56 & 1 \\
\end{array}\right).$$
The general term of this matrix is
$$\frac{\binom{3n-2k}{k} \binom{3n-3k}{n-k}}{2n-2k+1}.$$ The first column entries of this matrix, namely the ternary numbers $\frac{1}{2n+1}\binom{3n}{n}$ \seqnum{A001764},  correspond to lattice paths from $(0,0)$ to $(n,0)$ with step set $\{(1,1), (-1,-2)\}$. This can be seen because we have
$$\frac{u}{x}=1+\frac{u^3}{x^2},$$ or, equivalently,
$$g=1+xg^3.$$ The corresponding left factor matrix is the ternary matrix $(t(x), xt(x))$ \seqnum{A110616} that begins
$$\left(
\begin{array}{ccccccc}
 1 & 0 & 0 & 0 & 0 & 0 & 0 \\
 1 & 1 & 0 & 0 & 0 & 0 & 0 \\
 3 & 2 & 1 & 0 & 0 & 0 & 0 \\
12 & 7 & 3 & 1 & 0 & 0 & 0 \\
55 & 30 & 12 & 4 & 1 & 0 & 0 \\
273 & 143 & 55 & 18 & 5 & 1 & 0 \\
1428 & 728 & 273 & 88 & 25 & 6 & 1 \\
\end{array}\right),$$
where
$$t(x)=\frac{2}{\sqrt{3x}} \sin\left(\frac{1}{3}\sin^{-1}\left(\frac{\sqrt{27x}}{2}\right)\right)$$ is the generating function of ternary numbers. See Figure \ref{Ter}. The ternary matrix satisfies the recurrence
$$t_{n,k}=t_{n-1,k-1}+t_{n+1,k+2},$$ where $t_{n,k}=0$ if $n<0$ or $k<0$ or $k>n$, and $t_{0,0}=1$. It is clear how these considerations may be extended to the study of Fuss-Catalan numbers and lattice paths.

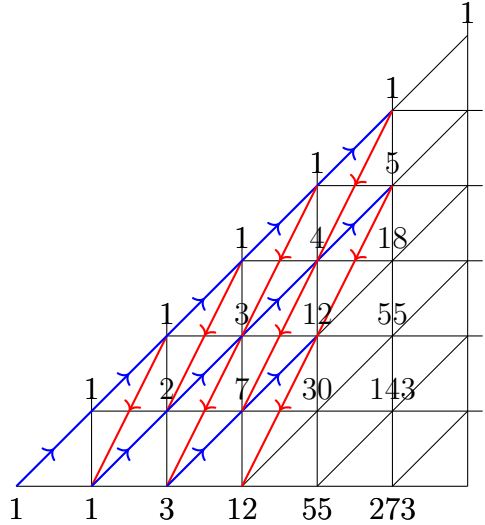
\begin{figure}
\begin{center}
\begin{tikzpicture}
\draw(0,0)--(6,0);
\draw(0,0)--(6,6);
\draw(1,0)--(1,1.2);
\draw(2,0)--(2,2.2);
\draw(3,0)--(3,3.2);
\draw(4,0)--(4,4.2);
\draw(5,0)--(5,5.2);
\draw(6,0)--(6,6.2);

\draw(1,1)--(6.2,1);
\draw(2,2)--(6.2,2);
\draw(3,3)--(6.2,3);
\draw(4,4)--(6.2,4);
\draw(5,5)--(6.2,5);

\draw(1,0)--(6,5);
\draw(2,0)--(6,4);
\draw(3,0)--(6,3);
\draw(4,0)--(6,2);
\draw(5,0)--(6,1);

\draw[red, thick, postaction={decorate},
      decoration={markings, mark=at position 0.5 with {\arrow{>}}}]
      (2,2) -- (1,0);

\draw[red, thick, postaction={decorate},
      decoration={markings, mark=at position 0.5 with {\arrow{>}}}]
      (3,3) -- (2,1);

\draw[red, thick, postaction={decorate},
      decoration={markings, mark=at position 0.5 with {\arrow{>}}}]
      (3,2) -- (2,0);

\draw[blue, thick, postaction={decorate},
      decoration={markings, mark=at position 0.5 with {\arrow{>}}}]
      (0,0) -- (1,1);

\draw[blue, thick, postaction={decorate},
      decoration={markings, mark=at position 0.5 with {\arrow{>}}}]
      (1,1) -- (2,2);

\draw[blue, thick, postaction={decorate},
      decoration={markings, mark=at position 0.5 with {\arrow{>}}}]
      (2,2) -- (3,3);

\draw[blue, thick, postaction={decorate},
      decoration={markings, mark=at position 0.5 with {\arrow{>}}}]
      (2,1) -- (3,2);

\draw[blue, thick, postaction={decorate},
      decoration={markings, mark=at position 0.5 with {\arrow{>}}}]
      (1,0) -- (2,1);

\draw[blue, thick, postaction={decorate},
      decoration={markings, mark=at position 0.5 with {\arrow{>}}}]
      (3,1) -- (4,2);

\draw[blue, thick, postaction={decorate},
      decoration={markings, mark=at position 0.5 with {\arrow{>}}}]
      (2,0) -- (3,1);

\draw[blue, thick, postaction={decorate},
      decoration={markings, mark=at position 0.5 with {\arrow{>}}}]
      (3,2) -- (4,3);

\draw[blue, thick, postaction={decorate},
      decoration={markings, mark=at position 0.5 with {\arrow{>}}}]
      (4,3) -- (5,4);

\draw[blue, thick, postaction={decorate},
      decoration={markings, mark=at position 0.5 with {\arrow{>}}}]
      (3,3) -- (4,4);

\draw[blue, thick, postaction={decorate},
      decoration={markings, mark=at position 0.5 with {\arrow{>}}}]
      (4,4) -- (5,5);

\draw[red, thick, postaction={decorate},
      decoration={markings, mark=at position 0.5 with {\arrow{>}}}]
      (5,4) -- (4,2);

\draw[red, thick, postaction={decorate},
      decoration={markings, mark=at position 0.5 with {\arrow{>}}}]
      (4,2) -- (3,0);

\draw[red, thick, postaction={decorate},
      decoration={markings, mark=at position 0.5 with {\arrow{>}}}]
      (4,4) -- (3,2);

\draw[red, thick, postaction={decorate},
      decoration={markings, mark=at position 0.5 with {\arrow{>}}}]
      (5,5) -- (4,3);

\draw[red, thick, postaction={decorate},
      decoration={markings, mark=at position 0.5 with {\arrow{>}}}]
      (4,3) -- (3,1);

\node at (0,-.3) {1};
\node at (1,1.3) {1};
\node at (2,2.3) {1};
\node at (3,3.3) {1};
\node at (4,4.3) {1};
\node at (5,5.3) {1};
\node at (6,6.3) {1};

\node at (1,-.3) {1};
\node at (2,-.3) {3};
\node at (3,-.3) {12};
\node at (4,-.3) {55};
\node at (5,-.3) {273};

\node at (0,-.3) {1};
\node at (1,1.3) {1};
\node at (2,2.3) {1};
\node at (3,3.3) {1};
\node at (4,4.3) {1};
\node at (5,5.3) {1};
\node at (6,6.3) {1};

\node at (1,-.3) {1};
\node at (2,-.3) {3};
\node at (3,-.3) {12};
\node at (4,-.3) {55};
\node at (5,-.3) {273};

\node at (2,1.3) {2};

\node at (3,1.3) {7};
\node at (3,2.3) {3};

\node at (4,1.3) {30};
\node at (4,2.3) {12};
\node at (4,3.3) {4};

\node at (5,1.3) {143};
\node at (5,2.3) {55};
\node at (5,3.3) {18};
\node at (5,4.3) {5};
\end{tikzpicture}
\end{center}
\caption{The ternary matrix $(t(x), xt(x))$ as a path matrix}\label{Ter}
\end{figure}

Note that by multiplying the matrix  $\mathbf{T}$ (on the right) by the inverse binomial matrix, we obtain a ternary analog of the Narayana matrix. This ternary analog begins
$$\left(
\begin{array}{ccccccc}
1 & 0 & 0 & 0 & 0 & 0 & 0 \\
0 & 1 & 0 & 0 & 0 & 0 & 0 \\
0 & 2 & 1 & 0 & 0 & 0 & 0 \\
0 & 4 & 7 & 1 & 0 & 0 & 0 \\
0 & 8 & 30& 16 & 1 & 0 & 0 \\
0 &16 & 104 & 122 & 30 & 1 & 0 \\
0 & 32 &320 & 660 & 365 & 50 & 1 \\
\end{array}\right).$$ This matrix has generating function
$$\frac{2}{\sqrt{3}}\sqrt{x(1-x(y-1))}\sin\left(\frac{1}{3}\sin^{-1}\left(\frac{3\sqrt{3}}{2} \frac{\sqrt{x(1-x(y-1))}}{(1-x(y-1))^2}\right)\right).$$ Setting $y=1$ shows that the row sums of this matrix are the ternary numbers. The embedded matrix $(t_{n,k})_{1 \le n,k \le \infty}$ is the matrix \seqnum{A091320}, which counts the number of non-crossing trees with $n$ edges and $k$ leaves.

\section{Conclusions} We have seen that many Riordan arrays can be associated to families of lattice paths. The rectification and triangulation of such matrices are associated with related families of lattice paths. For certain Riordan arrays, we can express their generating functions in a way that directly exhibits the defining parameters of the associated lattice paths. In these instances, there is also a direct link to a defining $A$-matrix.

\bigskip
\hrule

\noindent 2010 {\it Mathematics Subject Classification}:
Primary 05A1f5; Secondary 15B36, 11B37, 11B83.
\noindent \emph{Keywords:} Riordan array, lattice path, generating function, linear recurrence.

\bigskip
\hrule
\bigskip
\noindent (Concerned with sequences
\seqnum{A000045},
\seqnum{A000073},
\seqnum{A000108},
\seqnum{A000129},
\seqnum{A001045},
\seqnum{A001353},
\seqnum{A001764},
\seqnum{A002605},
\seqnum{A006190},
\seqnum{A006318},
\seqnum{A007482},
\seqnum{A080247},
\seqnum{A008288},
\seqnum{A026003},
\seqnum{A033184},
\seqnum{A039599},
\seqnum{A064189},
\seqnum{A090181},
\seqnum{A091320},
\seqnum{A110616},
\seqnum{A111961},
\seqnum{A143330},
\seqnum{A191649}, and
\seqnum{A321621}
).

\end{document}